\documentclass[12pt]{amsart}

\usepackage{amssymb}
\usepackage{amsmath}
\usepackage{amsfonts}
\usepackage[all]{xy} 

\newcommand{\cpo}{\mathbb{C}_p[[z]]}
\newcommand{\coneo}{\mathbb{C}_1[[z]]}
\newcommand{\cpinfty}{\mathbb{C}_p[[1/z]]}
\newcommand{\coneinfty}{\mathbb{C}_1[[1/z]]}
\newcommand{\cpounion}{\mathbb{C}_{\geq p}[[z]]}
\newcommand{\coneounion}{\mathbb{C}_{\geq 1}[[z]]}
\newcommand{\cpinftyunion}{\mathbb{C}_{\leq p}[[1/z]]}
\newcommand{\coneinftyunion}{\mathbb{C}_{\leq 1}[[1/z]]}
\newcommand{\cno}{\mathbb{C}_n[[z]]}
\newcommand{\cninfty}{\mathbb{C}_n[[1/z]]}
\theoremstyle{plain}
        \newtheorem{theorem}{Theorem}[section]
        \newtheorem{lemma}[theorem]{Lemma}
        \newtheorem{proposition}[theorem]{Proposition}
        \newtheorem{corollary}[theorem]{Corollary}

\theoremstyle{definition}
        \newtheorem{definition}[theorem]{Definition}
        \newtheorem{example}{Example}[section]

\theoremstyle{remark}
    \newtheorem{remark}[theorem]{Remark}
    \newtheorem{question}{Question}[section]

\numberwithin{equation}{section} 
\numberwithin{figure}{section} 

\usepackage{fullpage}

\author{Eric Schippers}
\title{A power matrix approach to the Witt algebra and Loewner equations}
\begin{document}
\maketitle

\begin{abstract} The theory of formal power series and derivation is developed
 from the point of view of the power matrix.  A Loewner equation for formal power series
 is introduced.  We then show that the matrix exponential is surjective onto the group
 of power matrices, and the coefficients are entire functions of finitely many coefficients
 of the infinitesimal generator.  Furthermore coefficients of the solution to the Loewner equation
 with constant infinitesimal generator can be obtained by exponentiating an infinitesimal power matrix.
 We also use the formal Loewner equations to investigate the relation between holomorphicity of an infinitesimal
 generator to holomorphicity of the exponentiated matrix.
\end{abstract}
\begin{section}{Introduction}
\begin{subsection}{Background and context}
 In conformal field theory, the Witt and Virasoro algebra are modelled using
 formal power series and derivations.  The purpose of this paper is
to outline an alternate approach to formal power series and
derivations which uses the power matrix in conjunction with the
Loewner equation.  We introduce here a notion of ``formal Loewner
equations'', which are Loewner equations for formal power series.
Even without the assumption of analyticity, univalence, or
subordination there are a formal Loewner partial and ordinary
differential equation which have the same form as the standard
Loewner equations.  Together with the power matrix, this provides
a convenient approach to the exponential map between derivations
and formal power series.

The connection of the Virasoro algebra (the central extension to
the Witt algebra) to Loewner theory was investigated by Markina,
Prokhorov and Vasil'ev \cite{MarkinaProkhorovVasilev}
\cite{MarkinaVasilev}. With a power matrix approach to the
semigroup of locally analytic and univalent maps near zero, the
Witt algebra has a matrix representation, as was observed by the
author in \cite{Ericpowermatrix}.  This provides a convenient way
to deal with the ``Lie theory'' of formal power series, and hence
also of the semigroup of locally univalent maps.  This point of
view is thoroughly developed in this paper, especially the
properties of the matrix exponential.

Results for formal power series and derivations, which can be
found in Huang \cite{Huang}, have consequences for the power
matrix and formal Loewner equations which are not known in
function theory. Although it would be possible to derive some of
the theorems for the power matrix directly from results in
\cite{Huang}, we give proofs here entirely in terms of the power
matrix and formal Loewner equations.  This has the advantages that
the paper is self-contained, and that it makes the results and
methods accessible to function theorists.  Furthermore, the
connection to Loewner theory suggests an approach to the problem
of whether the exponential map and its inverse preserve
holomorphicity, which we discuss in the last section.

 The power matrix is a well-known
 combinatorical object \cite{Comtet} which has been extensively studied.
 It is a useful tool appearing in several
 important constructions in geometric function theory, such as Faber polynomials
 and Grunsky matrices  \cite{Davis} \cite{Jabotinsky} \cite{SchifferTammi}.  The power matrix reveals
 non-trivial algebraic structure in Loewner theory \cite{SchifferTammi}.  The author
 gave a simple algebraic form for the Loewner-Schiffer differential system of a quadratic
 differential \cite{Ericpowermatrix}.  In the present setting it provides
 a convenient matrix representation of the Witt algebra.  In general, if a class
 of analytic functions can be represented as composition operators on a Hilbert
 or Banach space then the power matrix is the expression for this operator in a
 universally applicable basis $\{z^n\}$.  Thus the power matrix formalism should
 be of use in the representation theory of the Witt and Virasoro
 algebras on Hilbert spaces.
\end{subsection}
\begin{subsection}{Results and Outline}
 We will now describe the results.
 A concise list is given below.   In two-dimensional conformal field theory,
 the two ``halves'' of the Witt
 algebra are modelled as infinitesimal generators of locally
 analytic maps at $0$ and $\infty$, with a simple zero or pole respectively.
 A kind of Lie
 theory exists, where formal power series at $0$ and $\infty$
 stand in for the conformal maps, and derivations of power series
 correspond to the Lie algebra.  Derivations of power series at
 $0$ are one half of the Witt algebra, with basis
 \[  \left\{ z^n \frac{\partial}{\partial z} \,:\, n\geq 1\right\}. \]
 and derivations of power series at $\infty$ are the other half,
 with basis
 \[  \left\{ z^n \frac{\partial}{\partial z} \,:\, n \leq 1
 \right\}.  \]
 The development of this point of view can be found in
 \cite{Huang}.

 Our approach to the Witt algebra is to use the ``power matrix'' representation of power
 series.  The power matrix is defined as follows, in the case of locally univalent maps fixing
 $0$. Given a power series
 $f(z)=f_1 z + f_2 z^2 + \cdots$ the power
 matrix is the matrix of the linear transformation of power series
 given by $g \mapsto g \circ f$, represented in the basis
 $z^n$, $n \geq 1$.  Matrices of this form are denoted by $[f]$.
  In this representation, composition of power series
 corresponds to matrix multiplication.
 There is a
 Lie algebra of infinitesimal generators, corresponding to holomorphic functions $h$ which
 vanish at $0$, whose elements will be denoted by $\left<h\right>$. The ``Lie
 theory'' of the exponential map, Lie brackets, trivialization of
 the tangent bundle etc., takes a simple and convenient form.

As mentioned above, the theory of formal power series and derivations
has interesting consequences for function theory.
For example, we show that
 \begin{enumerate}
  \item The power matrix of any power series (in particular, a holomorphic
  power series) is the
  matrix exponential of $\left<h\right>$ for some formal power
  series $h$.
  \item If $\left<h\right>$ is the infinitesimal power matrix of a
  holomorphic function, then $\exp{t \left<h\right>}$ is the power
  matrix of a holomorphic function for some interval $[0,T]$.
  \item The coefficients of the
  solutions to the Loewner partial and ordinary differential
  equation, with constant infinitesimal generator $p$ and initial
  condition $f_0$, are given by
  $[f_0]\cdot\exp{t\left<zp\right>}$ or $\exp{t\left<zp\right>}\cdot[f_0]$
  respectively.
  \item Any finite set of coefficients of the power series of a solution to the Loewner partial or ordinary differential equation with constant infinitesimal generator can be uniformly approximated on a finite interval $[0,T]$ by the first row of a matrix polynomial.
  \item  For any formal power series $h$, each entry of
  $\exp{\left<h\right>}$ is an entire function of finitely many
  coefficients of $h$.
 \end{enumerate}
 The first two results are consequences of known results for formal power series and derivations.  The first
 result could be seen as a consequence of \cite[Proposition 2.1.7]{Huang}.  The second result is a consequence of a result of
  Huang \cite{Huang_private} communicated verbally to the author.  We give another
  proof using the power matrix and the formal
  Loewner equations.  Huang
  posed the question of whether
the exponential map (in the formal power series setting) preserves holomorphicity, in the sense that the exponential of a holomorphic power series must itself be holomorphic, and conversely whether if the exponential of a power series is
holomorphic, the power series must be holomorphic.  These questions should be of interest to function theorists, especially when
rephrased in terms of Loewner theory and the power matrix.  This is explored in the last section.

Here is an outline of the paper.  In Sections 2.1--2.3 we define the group of power matrices at $0$ and $\infty$ and their Lie algebras, derive their basic properties,
and outline the equivalence with the derivations and formal power series picture.  In Section 2.4, we introduce the formal Loewner equations for power series and power matrices.  Section 3 is devoted to establishing the properties of the matrix exponential.  In particular it is shown that the matrix exponential is surjective and nearly injective, and that its coefficients are entire functions of the generator.

 Section 4 is devoted to applications.  Section 4.1 shows how it is possible to compute or approximate the
solution to the Loewner partial or ordinary differential equations
through the use of matrix operations.  In 4.2 and 4.3 we
investigate the problem of whether the exponential and its inverse
preserve holomorphicity, by demonstrating its relation to
solutions of the Loewner differential equations.

 I am grateful to Yi-Zhi Huang and David Radnell for valuable discussions on the subject of this paper.
\end{subsection}
\end{section}
\begin{section}{The power matrix representation of local
coordinates}
\begin{subsection}{Spaces of formal power series}
 We will be working with formal power series at $0$ and $\infty$.
 First, we define the general spaces over $\mathbb{C}$
 \begin{equation} \label{de:cpo}
  \cpo=\left\{ \sum_{n=p}^\infty a_n z^n \,:\, a_n \in \mathbb{C}, \ a_p
  \neq 0 \right\}
 \end{equation}
 and
 \begin{equation} \label{de:cpinfty}
  \cpinfty=\left\{ \sum_{n=-\infty}^p a_n z^n \,:\, a_n \in \mathbb{C}, \ a_p
  \neq 0 \right\}.
 \end{equation}
 Of course these are not in general closed under addition or
 multiplication.  However
 \begin{equation} \label{de:cpounion}
  \cpounion= \cup_{n=p}^\infty \cno
 \end{equation}
 and
 \begin{equation} \label{de:cpinftyunion}
  \cpinftyunion = \cup_{n=p}^\infty \cninfty
 \end{equation}
 are vector spaces.

 The main sets of interest are $\coneo$ and $\coneinfty$.  These are
 groups under composition. In some sense the Lie algebra of $\coneo$ is
 the vector space of derivations
 \begin{equation} \label{de:derivationo}
  \mathfrak{d}(0) = \left\{ h(z) \frac{\partial}{\partial z} \,:\,
   h \in \coneounion \right\}
 \end{equation}
 \and the Lie algebra of $\coneinfty$ is
 \begin{equation} \label{de:derivationinfty}
  \mathfrak{d}(\infty) = \left\{ h(z) \frac{\partial}{\partial z} \,:\,
   h \in \coneinftyunion \right\}
 \end{equation}
 The bracket operation in both cases is
 \begin{equation} \label{de:bracketofderivations}
  \left[ h_1(z) \frac{\partial}{\partial z}, h_2(z)
  \frac{\partial}{\partial z} \right] = \left( h_1(z) h_2'(z) -
  h_1'(z) h_2(z) \right) \frac{\partial}{\partial z}.
 \end{equation}

 Clearly composition by $\coneo$ on the right is a group action on
 $\cpo$ and $\cpounion$ for all $p$.  Similarly composition by
 $\coneinfty$ on the right is a group action on $\cpinfty$ and
 $\cpinftyunion$ for all $p$.  In general, composition of elements of $\mathbb{C}_p[[z]]$ may involve
 infinite sums.  In all cases considered in this paper (for example, composition on the right by
 elements of $\mathbb{C}_1[[z]]$) only finite sums are involved and the definition of composition is
 unambiguous.
\end{subsection}
\begin{subsection}{The power matrices}
 \begin{definition}[power matrix at $0$] \label{de:powermatrixatoMP}  Let $f \in \cpo$. The power matrix $[f]$ of $f$ is defined to be the matrix whose entry in the $n$th row and $k$th column satisfies
 \[  f(z)^n = \sum_{k=pn}^\infty [f]^n_k z^k.  \]
 We denote
 \[  \mathcal{M}_p(0)=\{[f]\,:\, f\in \cpo\}  \]
 and
 \[  \mathcal{M}(0) = \mathcal{M}_1(0).  \]
 \end{definition}
 Note that the leading entry of each successive row of $[f]$ appears $p$ places to the right of the leading entry of the row immediately above.  (If $p<0$, it appears $|p|$ places to the left).
 Elements of $\mathcal{M}(0)$ are upper triangular.

  Functions which are meromorphic in a neighbourhood of infinity also have a power matrix representation.
 \begin{definition}[power matrix at $\infty$] The power matrix $[g]$ of an element $g \in \cpinfty$ is the matrix whose entries $[g]^n_k$ satisfy
 \[  g(z)^n=\sum_{k=-\infty}^{pn} [g]^n_k z^k.  \]
 Denote
 \[  \mathcal{M}_p(\infty) = \{ [g]\,:\, g \in \cpinfty \}  \]
 and
 \[  \mathcal{M}(\infty)= \mathcal{M}_{1}(\infty).  \]
 \end{definition}

 For $g \in \mathcal{M}_p(\infty)$, the final non-zero entry of the $n$th row is in the $pn$th column.  In each row, the final non-zero entry is $p$ places to the right of the final non-zero entry of the row immediately above (if $p<0$ then it is $|p|$ places to the left).  Elements of $\mathcal{M}(\infty)$ are lower triangular.

 As mentioned above, composition does not necessarily make sense for arbitrary elements of $\cpo$; in some cases the matrix multiplication involves infinite sums.  Similarly for elements of $\cpinfty$.
 However, the group action on $\coneo$ on $\cpo$ by composition on the right only involves finite sums.
 Furthermore, it corresponds to matrix multiplication.
 This is the reason for the utility of the power matrix.
 \begin{proposition}  $\mathcal{M}(0)$ is a group.
  The map $f \mapsto [f]$ preserves the group action of $\coneo$ on $\cpo$.
  That is, for $f \in \coneo$ and $g \in \cpo$.
  \[  [g \circ f]=[g][f].  \]
  In particular, the map from $\coneo$ to $\mathcal{M}(0)$ given by $f \mapsto [f]$ is a group homomorphism.

  Similarly, $\mathcal{M}(\infty)$ is a group and
  the map $f \mapsto [f]$ preserves the group action of $\coneinfty$ on $\cpinfty$
  in the sense that for $f \in \coneo$ and $g \in \cpo$.
  \[  [g \circ f]=[g][f].  \]
  The map from $\coneinfty$ to $\mathcal{M}(\infty)$ given by $f \mapsto [f]$ is a group homomorphism.
 \end{proposition}
 \begin{proof}  For the case of power series at $0$,
 consider the power series of $g \circ f$:
 \begin{align*}
 ( g \circ f )^n & = \sum_{k=pn}^\infty [g]^n_k f(z)^k
  = \sum_{k=pn}^\infty [g]_k^n \sum_{l=k}^\infty [f]^k_l z^l \\
 & = \sum_{l=pn}^\infty \left( \sum_{k=pn}^l [g]^n_k[f]^k_l \right)z^l.
 \end{align*}
 The case $p=1$ shows that $\mathcal{M}(0)$ is closed under multiplication
  and $f \mapsto [f]$ is a group homomorphism.  Clearly every element $[f]$ has the inverse $[f^{-1}]$
  by the above formula.  The general case $p \neq 1$ establishes the other claims.

  The claim for power series at $\infty$ follows from a similar computation:
   \begin{align*}
   (g \circ f)^n(z) & = \sum_{k=-\infty}^{pn} [g]^n_kf(z)^k = \sum_{k=-\infty}^{pn} [g]^n_k \sum_{l=-\infty}^k [f]^k_l z^l \\
   & = \sum_{l=-\infty}^{pn} \left( \sum_{k=pn}^l [g]^n_k [f]^k_l \right) z^l
  \end{align*}
 \end{proof}
  Note that there are
  no infinite sums.
  Furthermore if $f^{-1}$ is the inverse of $f \in \coneo$ in
  a neighbourhood of the origin, then $[f^{-1}]=[f]^{-1}$.  Similarly for elements of $\coneinfty$.

 The matrix groups $\mathcal{M}(0)$ and $\mathcal{M}(\infty)$ satisfy certain algebraic relations among the rows.
 \begin{proposition}
  If a doubly-infinite upper triangular matrix is in $\mathcal{M}(0)$
  then
  \[ n[f]^m_n = \sum_{l=m-1}^{n-1} m (n-l) [f]^{m-1}_l[f]^1_{n-l}.  \]
    A doubly-infinite lower triangular matrix
    in $\mathcal{M}(\infty)$
  satisfies the relations
  \[ n[f]^m_n = \sum_{l=n-1}^{m-1} m (n-l) [f]^{m-1}_l[f]^1_{n-l}.  \]
 \end{proposition}
 \begin{proof}  For elements of $\mathcal{M}(0)$,
 we have by multiplying the power series of $mf^{m-1}$ and $f'$ that
 \[   m f^{m-1} f' = \sum_{n=m-1}^\infty \sum_{l=m-1}^n m(n-l+1)[f]^{m-1}_l [f]^1_{n-l+1} z^n.  \]
 The relations can be obtained by equating the coefficients of the above series
 with
 \[  \frac{\partial}{\partial z} f^m = \sum_{n=m-1}^\infty (n+1)[f]^m_{n+1} z^n. \]

  The case of $\mathcal{M}(\infty)$ can be established by equating the coefficients of
  \[  f^{m-1}(z) f'(z) = \sum_{n=-\infty}^{m-1}  \sum_{l=n}^{m-1} [f]^{m-1}_l [f]^1_{n-l+1} z^n  \]
  and
  \[  \frac{\partial f^m}{\partial z}(z) = \sum_{n=-\infty}^{m-1} (n+1) [f]^m_{n+1} z^n.  \]
 \end{proof}
  \begin{remark}  \label{re:power_rows_polynomial}
  In particular, if $[f] \in \mathcal{M}(0)$ then $[f]^m_n$ is a polynomial in $[f]^1_k$ for $k=1,\ldots, n-m+1$, so long as $m \geq 0$.
  If $[f] \in \mathcal{M}(\infty)$ then $[f]^m_n$ is a polynomial in $[f]^1_k$ for $k=n-m+1,\ldots,1$, so long as $m \geq 0$.
 \end{remark}
\end{subsection}
\begin{subsection}{The Lie algebras of the groups of power matrices}
  Next we give a representation of the Lie algebra of
  $\mathcal{M}(0)$.  Consider the
 tangent space at the identity of this matrix group, using the
 variation
 \begin{equation} \label{standardvariation}
  G_\lambda(z)= z + \lambda h(z) + o(\lambda),
 \end{equation}
 where $\lambda$ is a real parameter, and $h(z)$ is a holomorphic function
 in a neighbourhood of $0$.  We assume that $G_\lambda(0)=0$ and $h(0)=0$.
  We have by an easy computation
 \[  \left. \frac{d}{d\lambda} \right|_{\lambda=0} \left[G_\lambda
     \right]^m_n = m [h ]^1_{n-m+1}.  \]
 So it is reasonable to make the following definition.
 \begin{definition} \label{de:Liealgebra_zero}
  The infinitesimal power matrix of $h \in \coneounion$ is the infinite upper triangular
  matrix $\left<h\right>$ whose entry in row $m$ and column $n$ is $m h_{n-m+1}$.  Denote
  \[  \mathfrak{m}(0)= \{ \left< h\right> \,:\, h \in \coneounion \}.  \]
 \end{definition}

 Explicitly,
 \begin{equation} \label{eq:Liealgebraexpression_zero}
  \left<h\right> = \left(
  \begin{array}{cccccccc}
   & \vdots & \vdots & \vdots & \vdots & \vdots &  \\
   \cdots & -h_1 & - h_2 & - h_3 & - h_4 & h_5 & \cdots \\
   \cdots & 0 & 0 & 0 & 0 & 0 & \cdots \\
   \cdots & 0 & 0 & h_1 & h_2 & h_3 & \cdots \\
   \cdots & 0 & 0 & 0 & 2h_1 & 2h_2 & \cdots \\
   \cdots & 0 & 0 & 0 & 0 & 3h_1 & \cdots \\
    & \vdots & \vdots & \vdots & \vdots & \vdots & \\
  \end{array}
 \right).
 \end{equation}

Next we describe the Lie algebra of $\mathcal{M}(\infty)$.  Again, consider a one-parameter family of functions in $\mathcal{M}(\infty)$ given by say
\[ F_\lambda(z) = z + \lambda h(z) + o(\lambda) \]
for some $h \in \coneinftyunion$.   As above we must have that
 \[  \left. \frac{d}{d\lambda} \right|_{\lambda=0} \left[F_\lambda
     \right]^m_n = m [h ]^1_{n-m+1}.  \]
 However, matrices $\left< h \right>$ of this form are lower triangular.  Explicitly
  \begin{equation} \label{eq:Liealgebraexpression_infty}
  \left<h\right> = \left(
  \begin{array}{cccccccc}
   & \vdots & \vdots & \vdots & \vdots & \vdots &  \\
   \cdots & -h_1 & 0& 0 & 0 & 0 & \cdots \\
   \cdots & 0 & 0 & 0 & 0 & 0 & \cdots \\
   \cdots & h_{-1} & h_0 & h_1 & 0 & 0 & \cdots \\
   \cdots & 2h_{-2} & 2h_{-1} & 2h_0 & 2h_1 & 0 & \cdots \\
   \cdots & 3h_{-3} & 3h_{-2} & 3h_{-1} & 3h_0 & 3h_1 & \cdots \\
    & \vdots & \vdots & \vdots & \vdots & \vdots & \\
  \end{array}
 \right).
 \end{equation}
 \begin{definition} \label{de:Liealgebra}
  The power matrix of $h \in \coneinfty$ is the infinite lower triangular
  matrix $\left<h\right>$ whose entry in row $m$ and column $n$ is $m h_{n-m+1}$ .
   These matrices will be denoted by
  $\left< h \right>$.
  Denote
  \[  \mathfrak{m}(\infty)= \{ \left< h\right> \,:\, h \in \coneinftyunion \}.  \]
 \end{definition}

 The Lie algebras have obvious special bases. Let
 \begin{equation} \label{eq:edefinition}
  \mathbf{e}_k = \left<z^{k+1} \right>.
 \end{equation}
 Clearly
 $\{ \mathbf{e}_k : k \geq 0\}$
 forms a basis of $\mathfrak{m}(0)$ and
 $\{ \mathbf{e}_k : k \leq 0\}$
 forms a basis of $\mathfrak{m}(\infty)$.
 The matrix of $\mathbf{e}_k$ is
 \[  \left< z^{k+1}\right>^m_l = \left\{ \begin{array}{rr} m & m=l+k
   \\ 0 & \mbox{otherwise}. \end{array} \right. \]
 It's not hard to check that the matrix Lie bracket is
 \[  [ \mathbf{e}_k,\mathbf{e}_l ]= (l-k)\mathbf{e}_{k+l}.  \]

 This agrees with the Lie bracket in $\mathfrak{d}(0)$ and $\mathfrak{d}(\infty)$.
 \begin{proposition}
  The map
  \begin{align*}
   T : \mathfrak{d}(0) & \rightarrow  \mathfrak{m}(0)\\
    h \frac{\partial }{\partial z} & \mapsto  \left<h\right>
  \end{align*}
  is a Lie algebra isomorphism.  This also holds for
  $\mathfrak{m}(\infty)$ and $\mathfrak{d}(\infty)$.
 \end{proposition}
 \begin{proof} The map is clearly injective and surjective.
  So it suffices to check that the Lie brackets agree.  We have
  \[  \left[ h_1(z)\frac{\partial}{\partial z},h_2(z)\frac{\partial}{\partial z} \right]= \left(h_1(z)
  h_2'(z) - h_1(z) h_2'(z) \right) \frac{\partial}{\partial z}.  \]
  Thus on the basis $z^{n+1} \partial/\partial z$ we have
  \[  T \left[ z^{n+1}\frac{\partial}{\partial z},z^{k+1}\frac{\partial}{\partial z} \right]= T \left( (k-n) z^{k+n+1} \frac{\partial}{\partial z}\right) = (k-n) \left<z^{n+k+1} \right> = \left[ \left<z^{n+1}\right>, \left<z^{k+1} \right>\right] \]
  which proves the claim.
 \end{proof}

 The following theorem allows one to recognize when expressions in terms of functions have a simple matrix form.  Special cases were given in \cite{Ericpowermatrix}.
 \begin{theorem} \label{th:leftandrightmultiplication}
  Let $g \in \cpo$, $f \in \coneo$ and $h \in
  \coneounion$.  Then
  \begin{equation*}
  \left[ m g^{m-1} \circ f \cdot g' \circ f \cdot h \circ f \right]^1_n =
  \sum_{l,k} \left[ g \right]^m_l
  \left< h \right>^l_k \left[f \right]^k_n.
 \end{equation*}
 The same formula holds for $g \in \cpinfty$, $f \in
 \coneinfty$ and $h \in \coneinftyunion$.
 \end{theorem}
 \begin{proof} We prove the claim for power series at $0$.
  Let $F_\lambda(z)= z + \lambda h(z) + o(\lambda)$.
  \begin{align*}
   \left.\frac{d}{d\lambda} \right|_{\lambda=0} g^m \circ
   F_\lambda \circ f & \left. = m g^{m-1} \circ F_ \lambda \circ f \cdot g' \circ F_\lambda \circ f  \cdot
   \frac{dF_\lambda}{d\lambda}    \circ f\right|_{\lambda=0}
   \\
   & = m g^{m-1} \circ f \cdot h \circ f \cdot g'  \circ f.
  \end{align*}
  On the other hand, by the linearity of matrix multiplication
  \[ \left.\frac{d}{d\lambda} \right|_{\lambda=0} g(z)^m =
    \left.\frac{d}{d\lambda} \right|_{\lambda=0} [g]^m_k
    [F_\lambda]^k_l[f]^l_n z^n = [g]^m_k\left<h \right>^k_l [f]^l_n z^n.  \]
  Comparing the coefficients proves the claim.  The same proof works
  for power series at $\infty$.
 \end{proof}
 \begin{corollary} \label{co:justleftandjustright}
 Let $g \in \cpo$, $f \in \coneo$ and $h \in
 \coneounion$.  Then
 \begin{equation} \label{eq:leftmultiplication}
  \left[ m g^{m-1} \, g' \, h \right]^1_n =  \sum_k \left[ g \right]^m_k
  \left<  h \right>^k_n
 \end{equation}
 and
 \begin{equation} \label{eq:rightmultiplication}
  \left[ m f^{m-1} \cdot h \circ f \right]^1_n = \sum_k \left< h
  \right>^m_k \left[ f \right]^k_n.
 \end{equation}
 The same formulas hold for $g \in \cpinfty$, $f \in
 \coneinfty$ and $h \in \coneinftyunion$.
 \end{corollary}
 \begin{proof}  The first equation follows by choosing $f(z)=z$ in
 Theorem \ref{th:leftandrightmultiplication}.  The second equation
 follows by choosing $g(z)=z$.  This works both at $0$ and at
 $\infty$.
 \end{proof}
 \begin{remark}  A special case of these was derived in
 \cite{Ericpowermatrix}.  However the version of
 (\ref{eq:rightmultiplication}) given there contains an error.
 \end{remark}

 It follows immediately from
 Corollary \ref{co:justleftandjustright}, that application of a
 derivation $h(z) d/dz$ of a power series corresponds to multiplication on the
 right by $\left<h \right>$.
 \begin{proposition} \label{pr:right_mult_is_derivation}
  Let $h \in \coneounion$ and $g \in \cpo$.  Then
  \[  \left[ h(z) \frac{\partial}{\partial z} g(z)^m \right]^1_k
    = \sum_l [g]^m_l \left<h\right>^l_k.\]
  The same formula holds for $h \in \coneinftyunion$ and
  $g \in \cpinfty$.
 \end{proposition}
\end{subsection}
\begin{subsection}{Derivations and the formal Loewner differential equations}
   We show here that a version of the Loewner equations hold for formal power series.

 There are two ways to write the tangent vector to a formal power series.  One of these is a derivation.
 \begin{proposition} \label{pr:formalLoewnerequations}
  Let $f_t \in \coneo$ for $t \in (a,b)$ have coefficients which are differentiable in $t$.
  \begin{enumerate}
   \item There's an $h_t \in \coneounion$ such that
   \begin{equation}  \label{eq:formal_Loewner_partial}
    \frac{\partial f_t}{\partial t}(z)= h_t(z) f_t'(z).
   \end{equation}
   The $n$th coefficient of $h_t$ depends only on the first $n$ coefficients of $f_t$ and ${\partial f_t}/{\partial t}$.
   \item There's an $\tilde{h}_t \in \coneounion$ such that
   \begin{equation} \label{eq:formal_Loewner_ordinary}
    \frac{\partial f_t}{\partial t}(z)= \tilde{h}_t \circ f_t(z).
   \end{equation}
   The $n$th coefficient of $\tilde{h}_t$ depends only on the first $n$ coefficients of $f_t$ and ${\partial f_t}/{\partial t}$.
  \end{enumerate}
   Similarly, if $f_t \in \coneinfty$ has differentiable coefficients, one can find $h_t, \tilde{h}_t \in \coneinftyunion$ such that $f_t$, $h_t$ and $\tilde{h}_t$ satisfy equations (\ref{eq:formal_Loewner_partial}) and (\ref{eq:formal_Loewner_ordinary}).
 \end{proposition}
 \begin{proof} Observe that $\partial f_t/\partial t$ is in $\cpo$ for some $p \geq 1$.  To prove the first claim, formally set
 \[  h_t(z)= \frac{1}{f_t'(z)} \frac{\partial f_t}{\partial t}(z)  \]
 and rewrite this as a power series in $\cpo$.  This power series satisfies the differential equation, and it is easily seen that the $n$th coefficient of $h_t$ only depends on the first $n$ coefficients of $f_t$ and $\partial f_t/\partial t$.

 To prove the second claim, set
 \[  \tilde{h}_t(z) = \frac{ \partial f_t}{\partial t} \circ f_t^{-1}(z).  \]
 Since $\coneo$ acts on $\cpounion$ on the right, the claim follows.  It is again easy to see that the $n$th coefficient of $h_t$ depends only on the first $n$ coefficients of $f_t$ and  $\partial f_t/\partial t$.
 \end{proof}
 In the standard Loewner theory, $h_t = z p_t$ and $\tilde{h}_t=-zp_t$ where
 $p_t(z)=1 + p_1 z + p_2 z^2 + \cdots \in \mathcal{P}$
 where
 \begin{equation} \label{eq:Pdefinition}
  \mathcal{P} =\{ p:\mathbb{D} \rightarrow \mathbb{C} \,:\, p \
  \mbox{holomorphic}, \mbox{Re}(p) >0, \ p(0)=1\}.
 \end{equation}
 The Loewner partial differential equation is
 \begin{equation} \label{eq:Loewner_PDE_definition}
  \frac{\partial f_t}{\partial t}(z) = z p_t(z) f_t'(z)
 \end{equation}
 and the Loewner ordinary differential equation is
 \begin{equation} \label{eq:Loenwer_ODE_definition}
  \frac{\partial f_t}{\partial t}(z) = -f_t(z) p_t \circ f_t(z).
 \end{equation}
 The ordinary differential equation, with suitable initial conditions,
 is guaranteed to have a solution which is holomorphic and univalent
 on the entire disc \cite{Pommerenke}.  The partial differential equation might not have
 an analytic solution on the disc.  Note that the choice $h(z)=zp(z)$ is motivated by the fact that $h \in \coneounion$.

 We will refer to equation (\ref{eq:formal_Loewner_partial}) as the formal Loewner partial differential equation and (\ref{eq:formal_Loewner_ordinary})  as the formal Loewner ordinary differential equation.  The term ``formal'' denotes the relation to formal power series.  Note that both the ordinary and partial formal Loewner equations are perfectly reasonable infinite systems of ordinary differential equations.
 \begin{remark} \label{re:loewner_finiteness}
 Furthermore, both systems of ordinary differential equations have the property that for each $n$ the system of equations for the first $n$ coefficients of $f_t$ is a {\it finite} system, and thus with suitable regularity solutions can be guaranteed on some interval.
 However the interval of existence depends on $n$, so there is no guarantee that there is an interval on which the infinite system has a solution.
 \end{remark}
 \begin{remark} \label{re:Loewner_natural}
 In \cite[Remark 5]{Ericpowermatrix}, the author made the observation that
 the form of the Loewner equations have little to do with the theory
 of subordination chains.  Here we see that the form of the Loewner equations arises naturally
 even in the very general setting of formal power series.
 \end{remark}
 \begin{remark} \label{re:Loewner_zero_infty_transform}
  There is a natural transformation between power series at $0$ and
  $\infty$, which is respected by the ordinary and partial Loewner
  differential equations.
  That is, we have the transformations
  \begin{eqnarray*}
   \coneo & \rightarrow & \coneinfty \\
   f(z) & \mapsto & 1/f(1/z)
  \end{eqnarray*}
  and
  \begin{eqnarray*}
   \coneounion & \rightarrow & \coneinftyunion \\
   h(z) & \mapsto & -z^2h(1/z).
  \end{eqnarray*}
  It's not hard to see that $f_t$ and $h_t$ satisfy the Loewner
  ordinary differential equation at $0$ if and only if
  $\tilde{f}_t(z)=1/f_t(1/z)$ and $\tilde{h}_t(z)=-z^2h_t(1/z)$
  satisfy the Loewner ordinary differential equation at $\infty$.
  The same statement holds for the Loewner partial differential
  equation.
 \end{remark}

 There is of course a matrix version of the formal Loewner equations.
 \begin{proposition} \label{pr:formal_matrix_Loewner_equal}
  \begin{enumerate}
   \item $f_t \in \coneo$ and $h_t \in \coneounion$ satisfy the formal Loewner partial differential equation on an interval $(a,b)$ if and only if their power matrices satisfy
       \[  \frac{d}{dt} [f_t] = [f_t] \left< h_t \right>  \]
       on $(a,b)$.
   \item $f_t \in \coneo$ and $h_t \in \coneounion$ satisfy the formal Loewner ordinary differential equation on an interval $(a,b)$ if and only if their power matrices satisfy
       \[  \frac{d}{dt} [f_t] = \left<h_t \right> [f_t].  \]
  \end{enumerate}
  The same claim holds for the formal Loewner differential equations at $\infty$.
 \end{proposition}
 \begin{proof}  Assume that $f_t$ satisfies the formal Loewner partial differential equation.  It follows that
 \[  \frac{\partial}{\partial t} f_t^m = m f^{m-1} \frac{\partial f_t}{\partial t} = m f^{m-1} f_t' h_t.  \]
 By Corollary \ref{co:justleftandjustright} it follows that
 \[  \frac{\partial}{\partial t} \left[f_t \right]^m_k = \sum_n [f_t]^m_n \left<h\right>^n_k  \]
 so $[f_t]$ and $\left<h_t\right>$ satisfy the matrix Loewner partial differential equation.  Conversely, assume that $[f_t]$ and $\left<h_t\right>$ satisfy the matrix Loewner partial differential equation.  Again by Corollary \ref{co:justleftandjustright}, the first row of the matrix equation is equivalent to
 \[  \frac{\partial}{\partial t} f_t = f_t' h_t.  \]

 Now assume that $f_t$ and $h_t$ satisfy the formal Loewner ordinary differential equation.  In this case,
 \[  \frac{\partial}{\partial t} f_t^m = m f^{m-1} \frac{\partial f_t}{\partial t} = m f^{m-1} h_t \circ f_t. \]
 By Corollary \ref{co:justleftandjustright},
 \[  \frac{\partial}{\partial t} \left[f_t \right]^m_k = \sum_n \left<h\right>^m_n [f_t]^n_k \]
 so the matrix Loewner ordinary differential equation is satisfied.  Conversely, if the matrix Loewner ordinary differential equation is satisfied, then the first row of this equation is equivalent to the formal Loewner ordinary differential equation by Corollary \ref{co:justleftandjustright}.
 \end{proof}
 \begin{remark} \label{re:matrix_Loewner_ae}
  Of course, Proposition \ref{pr:formal_matrix_Loewner_equal} can be rephrased with weakened regularity, e.g. the matrix Loewner equations hold almost everywhere if and only if the formal Loewner equations hold almost everywhere, etc.
 \end{remark}
 \begin{remark} In the power matrix picture, we can clearly see that the two Loewner equations arise from the two possible trivializations of the tangent bundle to $\mathcal{M}(0)$ or $\mathcal{M}(\infty)$; that is by left or right multiplication.
 \end{remark}
\end{subsection}
\end{section}
\begin{section}{The exponential and logarithm}
  The exponential map connects the Lie group to its Lie algebra; for
 matrix groups, the exponential map is simply the matrix
 exponential.  In the present situation, although we have not
 discussed differentiable structures, the exponential map is
 nevertheless well-behaved and has all the desired algebraic
 properties.

 In this section we demonstrate that the exponential map is in some sense finite.  That is, the $k$th entry in the first row of $\exp\left<h \right>$ depends only on the first $k$ coefficients of $h$.  Furthermore, it is possible to solve for the first $k$ coefficients of $h$ in terms of the first $k$ coefficients of $\exp \left<h \right>$.
 The matrix exponential always converges in
 the sense that each coefficient is a convergent infinite sum, and
 is an entire function of the coefficients of $h$.
 Furthermore, we show that the matrix exponential is onto, and in a
 sense nearly one-to-one.
\begin{subsection}{The exponential map}
 We begin with some elementary observations.  By the $k$th diagonal of a matrix $A$, we mean the set of entries in the locations $A^m_{m+k}$.
 \begin{definition} A matrix $A$ is said to be $k$-diagonal if all of its entries are zero with the exception of the $k$th diagonal.  That is, $A^m_n=0$ unless $n=m+k$.
 \end{definition}
 \begin{example} $\mathbf{e}_k= \left< z^{k+1} \right>$ is $k$-diagonal for all integers $k$.
 \end{example}
 \begin{definition}
  For a doubly infinite matrix $A$, the $(m,n)$th principal block of $A$ is the matrix $[A]^i_j$ with $m\leq i \leq n$ and $m \leq j \leq n$.
 \end{definition}
 The following proposition guarantees that one can perform the usual
 matrix operations without worrying about convergence.
 \begin{proposition} \label{pr:multiplying_finite}
  Let $A$ and $B$ be doubly-infinite matrices.  Assume that $A$ and
  $B$ are both upper triangular, or that $A$ and $B$ are both lower
  triangular (in the sense that $A^i_j = 0$ if $j<i$ or $i <j$
  respectively).  Then the $(m,n)$th principal block of $AB$ is the
  product of the $(m,n)$th blocks of $A$ and $B$.
 \end{proposition}
 \begin{proof}
  Assume $A$ and $B$ are infinite upper-triangular matrices.  Let $\tilde{A}$ and $\tilde{B}$
  denote the $(m,n)$ principal blocks of $A$ and $B$ respectively.  It is clear that for all $n \leq k \leq m$ and $n \leq l \leq m$ the entry $[AB]^k_l$ involves only elements of $\tilde{A}$ and $\tilde{B}$, since
  \[  [AB]^k_l = \sum_{i=k}^l [A]^k_i [B]^i_l.  \]
  The proof of the other case is identical.
 \end{proof}
 Since the exponential of any finite square matrix converges
 \cite{Curtis}, we immediately have the following corollary.
 \begin{corollary} \label{co:exp_converges}Let $h \in \coneounion$.  Then each
 coefficient of the matrix exponential $\exp{\left<h\right>}$
 converges.  Further more the principal $(m,n)$ block of
 $\exp{\left<h\right>}$ is the exponential of the principal $(m,n)$
 block of $\left<h\right>$.  The same claims hold if $h \in \coneinftyunion$.
 \end{corollary}

 Two further essential properties of the matrix exponential carry over from the finite case for the same reason.
 \begin{proposition} \label{pr:expth_is_subgroup}
  For any $\left<h\right> \in \mathfrak{m}(0)$,
  \[  \mbox{exp}\left(t\left<h\right>\right) \mbox{exp} \left(s\left<h\right>\right)= \mbox{exp}\left((t+s)\left<h \right>\right).  \]
  The same claim holds for $\left<h\right> \in \mathfrak{m}(\infty)$.
 \end{proposition}
 \begin{proof}
  Since the conclusion of the Proposition is true for finite square matrices \cite{Curtis}, the claim follows.
 \end{proof}

 \begin{proposition} \label{pr:derivative_of_exp}
  For any $\left<h\right> \in \mathfrak{m}(0)$,
  \[  \frac{d}{dt} \mbox{exp}(t \left<h\right>) = \left<h\right> \mbox{exp}(t\left<h\right>) = \mbox{exp}(t\left<h\right>) \left<h\right>. \]
  The same claim holds for $\left<h\right> \in \mathfrak{m}(\infty)$.
 \end{proposition}
 \begin{proof}
 As in the proof of the previous Proposition, the idea is to use the facts that the conclusion holds for finite matrices and multiplication of the infinite upper triangular matrices is essentially finite.

 We prove the first equality.  Let $\delta^l_k$ denote the Kronecker delta function, which is $1$ if $l=k$ and $0$ otherwise.
 \begin{align*}
  \frac{d}{dt} [\mbox{exp}(t \left<h\right>)]^n_k &= \lim_{s \rightarrow 0} \frac{1}{s} \left( [\mbox{exp}((t+s) \left<h\right>)]^n_k -[\mbox{exp}(t \left<h\right>)]^n_k \right) \\
  & = \lim_{s \rightarrow 0} \sum_{l=n}^k [\mbox{exp}(t \left<h\right>)]^n_l \frac{1}{s} \left( [\mbox{exp}(s\left<h\right>)]^l_k -\delta^l_k \right) \\
  & = \sum_{l=n}^k [\mbox{exp}(t \left<h\right>)]^n_l \lim_{s \rightarrow 0} \frac{1}{s} \left( [\mbox{exp}(s\left<h\right>)]^l_k -\delta^l_k \right)
 \end{align*}
 The interchange of the limit with the sum is possible because the sum is finite.

 The right-most limit clearly exists and
 \begin{align*}
  \lim_{s \rightarrow 0} \frac{1}{s} \left( [\mbox{exp}(s\left<h\right>)]^l_k -\delta^l_k \right) & = \lim_{s \rightarrow 0} \left[ \sum_{m=1}^\infty \frac{1}{m!} s^{n-1} \left<h\right>^n \right]^l_k \\
  & = \left<h \right>^l_k.
 \end{align*}
 The second equality follows in the same way by factoring $\exp{\left<h\right>}$ to the right-hand side.
 \end{proof}

 The exponential maps $\mathfrak{m}(0)$ into $\mathcal{M}(0)$.
 \begin{theorem} \label{th:exp_is_power_matrix} If $f \in \coneo$ is given by
  \[ f(z)=\sum_{k=1}^\infty \left[ \exp{\left< h \right>} \right]^1_k z^k \]
  then
  \[  [f]=\exp{\left<h\right>}.  \]
  That is, $\exp{\mathfrak{m}(0)} \subset \mathcal{M}(0)$.  Similarly, if $f \in \coneinfty$ is given by
  \[   f(z)=\sum_{k=-\infty}^1 \left[ \exp{\left< h \right>} \right]^1_k z^k \]
  then
  \[  [f]=\exp{\left<h\right>}  \]
  and thus $\exp{\mathfrak{m}(\infty)} \subset \mathcal{M}(\infty)$.
 \end{theorem}
 \begin{proof}
  $\exp{t\left<h\right>}$ satisfies the matrix Loewner partial differential equation by Proposition \ref{pr:derivative_of_exp}.
  Let
  \[  f_t(z)= \sum_{k=1}^\infty \left[ \exp{t \left<h\right>} \right]^1_k z^k.  \]
  $f_t$ satisfies the formal Loewner partial differential equation for $t \geq 0$
  since by Definition \ref{de:Liealgebra_zero}
  \[  \frac{\partial f_t}{\partial t}(z)=  \sum_k \sum_l \left[ \exp{t\left<h\right>} \right]^1_l \left<h\right>^l_k z^k
  = \sum_k \sum_l l \left[ \exp{t\left<h\right>} \right]^1_l \left<h\right>^1_{k-l+1} z^k  \]
  and the right hand side is the series $f_t'(z) h(z)$.  By Proposition \ref{pr:formal_matrix_Loewner_equal} the power matrix $[f_t]$ of $f_t$ satisfies the matrix Loewner partial differential equation for all $t \geq 0$.  Thus the coefficients $[f_t]$ and $\exp{t \left<h\right>}$ satisfy the same differential equation and initial condition.  On each principal block, this is a finite system of ordinary differential equations and thus has a unique solution.  Thus the two matrices must agree on each principal block, so they are equal everywhere. Setting $t=1$ proves the claim.
 \end{proof}

 The matrix exponential agrees with the exponential of
 a derivation.
 \begin{theorem} \label{th:exp_derivation_composition}
  For $h \in \coneounion$ and $g \in \cpo$,
  \[  \left[ \mbox{exp}\left( h \frac{\partial}{\partial z} \right) \cdot g
  \right] = [g]\mbox{exp} \left<h\right>.
  \]
  The same formula holds if $h \in \coneinftyunion$ and $g \in
  \cpinfty$.
 \end{theorem}
 \begin{proof}  We prove that
 \[  \left[ \left( h(z)\frac{\partial}{\partial z} \right)^n g \right]^1_k = \sum_l [g]^1_l \left( \left<h\right>^n \right)^l_k  \]
 by induction.  The initial case follows directly from Proposition \ref{pr:right_mult_is_derivation} choosing $m=1$.  Assuming that the formula is true for $n$, we have that
 \begin{align*}
    \left[ \left( h(z)\frac{\partial}{\partial z} \right)^{n+1} g \right]^1_k & =
 \sum_l  \left[ \left( h(z)\frac{\partial}{\partial z} \right)^n g \right]^1_l  \left<h \right>^l_k = \sum_{l,m} [g]^1_m \left( \left<h\right>^n \right)^m_l \left<h \right>^l_k \\
      & =  \sum_l [g]^1_l \left( \left<h\right>^{n+1} \right)^l_k. \\
 \end{align*}
 Thus for all $k \geq 1$
 \[ \left[ \exp{\left( h  \frac{\partial}{\partial z}\right)} g \right]^1_k = \sum_l [g]^1_l \left[ \exp{\left<h\right>} \right]^l_k.  \]
 The claim now follows from Theorem \ref{th:exp_is_power_matrix}.
 \end{proof}

\end{subsection}
\begin{subsection}{Holomorphicity, surjectivity and near-invertibility of the exponential
map}
 In order to prove that the exponential must converge, and matrix
 operations are finite, we have only used that the matrices in
 question are upper or lower triangular.  The special form of the
 matrices has not been used.
 In fact the special form of the matrices implies much more.  We
 will now show that the coefficients of the exponential $\exp{\left<h\right>}$
 are entire functions of the coefficients of $\left<h\right>$.
 Furthermore the exponential map is onto.

 We begin with another elementary observation:
 \begin{proposition} \label{pr:kdiagonalproducts}
  If $A$ is $k$-diagonal and $B$ is $l$-diagonal then $AB$ is $k+l$-diagonal.
 \end{proposition}
 \begin{proof}
  Assume that $n \neq m+k+l$.  The terms in the sum
  \[  (AB)^m_n = \sum_{j=-\infty}^\infty A^m_j B^j_n  \]
  are zero unless $j=m+k$, since $A^m_j=0$ otherwise.  But in this case $B^j_n=0$ because $j+l =m+k+l \neq n$ by hypothesis.
 \end{proof}

 \begin{lemma} \label{le:hpowers_polynomials}
  Let $h \in \coneounion$.   For all integers $n \geq 1$ and $k \geq 2$,
  \[ \left[ \left<h \right>^n \right]^1_k = n h_k h_1^{n-1} + \Phi^n_k(h_1,\ldots, h_{k-1})  \]
  where $\left<h \right>^n$ denotes the $n$th matrix power of $\left<h \right>$ and $\Phi^n_k$ is a polynomial.

  Similarly, if $h \in \coneinftyunion$ then for all integers $n \geq 1$ and $k \leq 0$,
  \[ \left[ \left<h \right>^n \right]^1_k = n h_k h_1^{n-1} + \Phi^n_k(h_1,h_0,\ldots, h_{k+1})  \]

  In either case,
  \[  \left[\left<h \right>^n \right]^1_1=h_1^n . \]
 \end{lemma}
 \begin{proof}
  Formally, $\left<h \right>^n=(h_1 \mathbf{e}_0 + h_2 \mathbf{e}_1 + \cdots)^n$.  We are concerned only with the $(k-1)st$ diagonal.   By Proposition \ref{pr:kdiagonalproducts}, $\mathbf{e}_{k_1} \cdots \mathbf{e}_{k_m}$ is $k_1 + \cdots + k_m$-diagonal.  Thus
  \[  \left[ \left<h \right>^n \right]^1_k = \left[ \sum_l \sum_{k_1 + \cdots + k_m=l-1} h_{k_1+1} \cdots h_{k_m+1} \mathbf{e}_{k_1} \cdots \mathbf{e}_{k_m} \right]^1_k=
  \sum_{k_1 + \cdots + k_m=k-1} h_{k_1+1} \cdots h_{k_m+1}.  \]
  This is clearly a polynomial in $h_1,\ldots, h_{k}$.  The only terms with $h_{k}$ appearing are those of the form $h_1^{n-1} h_k$, and there are precisely $n$ of these terms.  This proves the claim.

  The proof in the case of $h \in \coneinftyunion$ is similar.
 \end{proof}
 \begin{lemma} \label{le:h_bound}
  Let $h \in \coneounion$.  Let $M_k= \sup_{1 \leq n \leq k} |h_n|^{1/n}$.
  For all $k$, the $k$th entry of the first row of $\left<h\right>$
  satisfies the following bound:
   \[  \left|\left[\left< h \right>^n \right]^1_k \right| \leq k^{2n} M_k^{n+k-1}.  \]
 Similarly, the same bound is satisfied for $h \in \coneinftyunion$
 if we set $M_k = \sup_{k \leq n \leq 1} |h_n|^{1/n}$.
 \end{lemma}
 \begin{proof}
   Let $B$ be the infinite upper triangular matrix with entries
  $[B]^n_m=n$ if $m \geq n$ and $0$ otherwise, for $n,m=1 \ldots \infty$.  Each entry of $\left<h\right>$ in the $(1,k)$ principal block satisfies the bound
  \[  \left|\left<h \right>^l_m \right| \leq B^l_m M_k^{m-l+1}.  \]
 Let $H$ be the $k \times k$ matrix whose entries are given by $H^l_m=B^l_m M_k^{m-l+1}$.
 It is easily checked that for $n \geq 1$
 \[  \left[ H^n \right]^l_m=\left[B^n\right]^l_m M_k^{n+m-l}.  \]
 Thus
 \[ \left| \left[ \left<h \right>^n \right]^1_k \right| \leq \left[B^n \right]^1_k M_k^{n+k-1}.  \]

 We claim that $\left[B^n \right]^1_k \leq k^{2n}$.  The proof is by induction.  It is clearly true for $n =1$.  Assume it holds for $n$.  Then
 \[  \left[B^{n+1} \right]^1_k = \sum_{l=1}^k \left[B^n \right]^1_l B^l_k = \sum_{l=1}^k l\left[B^n \right]^1_l.  \]
 Since $\left[B^n \right]^1_l \leq l^{2n} \leq k^{2n}$ by the inductive hypothesis for each $l \leq k$, we have that
 \[  \left[B^{n+1} \right]^1_k \leq k^{2n} \left( \sum_{l=1}^k l \right) \leq k^{2n+2}.  \]

 We now have the estimate
 \[  \left|\left[\left< h \right>^n \right]^1_k \right| \leq k^{2n} M_k^{n+k-1}.  \]

 The proof of the case $h \in \coneinftyunion$ is similar.
 \end{proof}
 \begin{remark} The above estimate can clearly be improved, for
 example to
  \[  \left[ B^n \right]^1_k \leq \left( \frac{k(k+1)}{2} \right)^n \]
  by replacing the estimate
  \[   \sum_{l=1}^k l \leq k^2  \]
  with the actual sum $k(k+1)/2$ in the inductive step.
 \end{remark}

 \begin{theorem} \label{th:universal_exp_functions}
  Let $k \geq 2$ and $h_1,\ldots, h_{k-1} \in \mathbb{C}$.  The series
  \[ \Psi_k(h_1,\ldots,h_{k-1}) = \sum_{n=1}^\infty \frac{1}{n!} \Phi_k^n(h_1,\cdots,h_{k-1}) \]
  converges absolutely and uniformly on compact subsets of
  $\mathbb{C}^{k-1}$.
  Similarly, for $k \leq 0$, and $h_{k+1},\ldots,h_1 \in \mathbb{C}$, the series
  \[ \Psi_k(h_{k+1},\ldots,h_1) = \sum_{n=1}^\infty \frac{1}{n!} \Phi_k^n(h_{k+1},\cdots,h_1) \]
  converges absolutely and uniformly in compact subsets of
  $\mathbb{C}^{1-k}$.  In particular, for all $k$ the functions $\Psi_k$ are
  entire in each variable.
 \end{theorem}
 \begin{proof} By Corollary $\ref{co:exp_converges}$ each coefficient of
  $\exp{\left<h\right>}$ converges, and by Lemma \ref{le:hpowers_polynomials}, the $k$th coefficient of the first row of $\exp{\left<h\right>}$ is a function only of $h_1,\ldots,h_k$.  By Lemma \ref{le:h_bound}, each element of the first row of the exponential $\exp{\left<h\right>}$ converges absolutely and uniformly on bounded sets in $\mathbb{C}^{k-1}$.    Thus by Lemma \ref{le:hpowers_polynomials} it
   follows that
   \[  \sum_{n=0}^\infty \frac{1}{n!} \left[ \left<h \right>^n \right]^1_k = \sum_{n=1}^\infty  \left(
   \frac{h_k h_1^{n-1}}{(n-1)!} + \Phi^n_k(h_1,\ldots, h_{k-1}) \right)  \]
   converges absolutely and uniformly on bounded sets in $\mathbb{C}^{k-1}$.  But the sum of the first term clearly converges to $h_k e^{h_1}$ from which the claim follows for $k \geq 2$.
   The proof in the case that $h \in \coneinftyunion$ is similar.
 \end{proof}
 \begin{corollary} \label{co:exp_expression}
  For $h \in \coneounion$ and $k \geq 2$,
  \[  \left[ \exp{\left<h\right>}\right]^1_k =  h_k e^{h_1} + \Psi_k(h_1,\ldots,h_{k-1}).  \]
  In particular, $\left|\exp{\left<h\right>}\right|^1_k$ is an
  entire function of each variable $h_1,\ldots, h_k$ for $k \geq
  1$.
  For $h \in \coneinftyunion$ and $k \leq 0$
  \[  \left[ \exp{\left<h\right>}\right]^1_k =  h_k e^{h_1} + \Psi_k(h_{k+1},\ldots,h_{1})  \]
  and $\left|\exp{\left<h\right>}\right|^1_k$ is entire in each
  variable $h_{k},\ldots,h_1$ for $k \leq 1$.
 \end{corollary}
 \begin{proof} This follows directly from Theorem
 \ref{th:universal_exp_functions} and Lemma
 \ref{le:hpowers_polynomials}, and the fact that $\Phi^n_k$
 are polynomials.  For $k=1$ the claim is immediate.
 \end{proof}
 \begin{remark} \label{re:expbound}
 Letting $M_k=\sup_{n\leq k} |h_n|^{1/n}$ as in Lemma \ref{le:h_bound}, the $k$th entry of the first row satisfies the following bound:
  \[  \left| \left[ \exp{\left< h \right>} \right]^1_k \right| \leq M_k^{k-1}e^{k^2M_k}. \]

  In particular, if $h$ is a convergent power series in a neighbourhood of $0$ then $M= \sup_k M_k$ exists and
   \[  \left| \left[ \exp{\left< h \right>} \right]^1_k \right| \leq M^{k-1}e^{k^2M}. \]
 \end{remark}
 From Corollary \ref{co:exp_expression} and Remark \ref{re:power_rows_polynomial}
 it immediately follows that
 \begin{corollary} \label{co:exp_entire} For $h \in \coneounion$ and $m \geq 1$,
 $[\exp{\left< h \right>}]^m_n$ is an entire function of $h_1,
 \ldots, h_{n-m+1}$.  For $h \in \coneinftyunion$ and $m \geq 1$,
 $[\exp{\left< h \right>}]^m_n$ is an entire function of
 $h_{n-m+1},\ldots, h_1$.
 \end{corollary}


 We can now establish the near-invertibility and surjectivity of the exponential function.
 \begin{theorem} \label{th:exp_invertible} Let $[f] \in \mathcal{M}(0)$.
 Let $h_1$ be some choice of $\log{[f]^1_1}$.
 There exists a unique $h \in \coneounion$ such that $h(z)=h_1z +
 \cdots$ and
 $[f] = \exp{\left<h\right>}$.  Furthermore, the coefficients $h_1,\ldots,h_k$
 are determined by the first $k$ coefficients $[f]^1_l$, $l=1,\cdots,k$.
 Similarly if $f \in \mathcal{M}(\infty)$, for a fixed choice of $h_1 =
 \log{[f]^1_1}$ there is a unique $h \in \coneinftyunion$ such that
 $h(z)=h_1z + \cdots$ and
 $[f]=\exp{\left<h\right>}$, and the coefficients $h_k,\ldots,h_1$
 are determined by $[f]^1_l$ for $l=k,\ldots,1$.

 Thus $\exp{\mathfrak{m}(0)} = \mathcal{M}(0)$ and $\exp{\mathfrak{m}(\infty)}= \mathcal{M}(\infty)$.
 \end{theorem}
 \begin{proof}  Assume $h \in \coneounion$.
  By Theorem \ref{th:universal_exp_functions} and Corollary
  \ref{co:exp_expression}
  it is clear that one can solve for the coefficients recursively via
  \[  \left< h\right>^1_{k+1}=e^{-h_1}\left([f]^1_{k+1} - \Psi_{k+1}(h_1,\ldots,h_k)  \right).  \]

  On the other hand if $h \in \coneinftyunion$ then one can solve
  for the coefficients recursively using
  \[  \left< h\right>^1_{k-1}=e^{-h_1}\left([f]^1_{k-1} - \Psi_{k-1}(h_k,\ldots,h_1) \right).  \]
 \end{proof}
\end{subsection}
\begin{subsection}{The matrix logarithm}
 Given $[f] = \exp{\left<h\right>}$,
 we may of course also find $\left<h\right>$ by computing the
 matrix logarithm of $[f]$, provided that it converges.  We show
 that it converges for $[f]^1_1=1$.
 \begin{theorem} Let $[f] \in \mathcal{M}(0)$ satisfy $[f]_1^1=1$.
 Let $I$ be the doubly infinite identity matrix $I=[z]$ and define
 \begin{equation} \label{eq:log_sum_definition}
  \log{[f]} :=  \sum_{n=1}^\infty \frac{(-1)^{n+1}}{n}
   ([f]-I)^n.
 \end{equation}
 The series converges to some $\left<h \right> \in \mathfrak{m}(0)$ (in the sense that each coefficient
 converges), and $h$ is the unique element of $\coneounion$
 satisfying $[f]= \exp{\left<h\right>}$ with first coefficient
 $h_1=1$.  Similarly, if $[f] \in \mathcal{M}(\infty)$ the above
 series converges to $\left<h\right> \in \mathfrak{m}(\infty)$ and
 $h$ is the unique element of $\coneinftyunion$ such that
 $[f]=\exp{\left<h\right>}$ with initial coefficient $h_1=1$.
 \end{theorem}
 \begin{proof}
  By Proposition \ref{pr:multiplying_finite} the $(m,n)$th
  principal block of $\log{[f]}$ is the corresponding series of
  the $(m,n)$th principal blocks of $[f]-I$.  This matrix is upper triangular
  (respectively lower triangular) with zero diagonal, and hence is
  nilpotent.  Thus the series is in fact finite.  Since the
  original series converges on each $(m,n)$ principal block, it
  must in fact converge.

  For finite matrices $\exp{\log{A}}=A$ so long as
  the logarithm converges \cite{Curtis}.  So by Corollary
  \ref{co:exp_converges} $\exp{\log{[f]}}=[f]$.  Since by Theorem
  \ref{th:exp_invertible} $\left<h\right>$ is uniquely determined,
  the claim follows.
 \end{proof}
 The proof of this theorem has a surprising consequence.
 \begin{theorem}
  Let $[f]\in \mathcal{M}(0)$ satisfy $[f]^1_1=1$.
  Let $h \in \coneounion$ be the unique series
  with first coefficient $h_1=1$ such that $[f]=\exp{\left<h\right>}$.
  For each $k \geq 1$, $h_k$ is a polynomial in
  $[f]^1_l$, $l=2,\ldots,k$.  Similarly if $[f] \in
  \mathcal{M}(\infty)$ satisfies $[f]_1^1$ and $h \in
  \coneinftyunion$ is the unique series such that $h_1=1$ and
  $[f]=\exp{\left<h\right>}$ then $h_k$ is a polynomial in
  $[f]^1_l$ for $l=k,\ldots,0$.
 \end{theorem}
 \begin{proof}  This follows from the finiteness of the sum
 (\ref{eq:log_sum_definition}).
 \end{proof}
\end{subsection}
\end{section}
\begin{section}{Applications}
\begin{subsection}{Computing coefficients of solutions to the Loewner equation}
 For power matrices, the one-parameter subgroups are generated by
 exponentials of $zp(z)$ for $p$ independent of time.  Similarly,
 the Loewner equation with time-independent $p$ appears in the analytic
 theory of semigroups \cite{Shoikhet}.  In this section we show
 that by the preceding results it is easy to
 compute the coefficients of a solution to either Loewner
 differential equation with the exponential map.

 By a Loewner chain we mean a one-parameter family of univalent functions $f_t$ on $\mathbb{D}$ for $t \in [a,\infty)$, satisfying the normalization $f_t(0)=0$, and $f_t'(0)=e^{\alpha t}$ for some constant $\alpha$,
 and the subordination condition $t \leq s \Rightarrow f_t(\mathbb{D}) \subset f_s(\mathbb{D})$.  Any Loewner chain satisfies the Loewner partial differential equation $\dot{f}_t(z)=zp_t(z)f_t(z)$ almost everywhere,  for some time-dependent $p_t$ which is measurable in $t$, complex analytic on $\mathbb{D}$ and satisfies $\mbox{Re}(p_t)>0$ for each $t$.  (We also must have that $\alpha=p(0)$).
 \begin{theorem} \label{th:LoewnerPDEsolved_by_exp}
  Let $f_t$ be a Loewner chain on $[a,\infty)$, with initial point $f_a$ and satisfying the Loewner partial differential equation $\dot{f}_t(z) = zp(z)f_t'(z)$ with constant infinitesimal generator.  Then the power matrix of $f_t$ is given by
 \[ [f_t]= [f_a]\exp{(t-a)\left<zp\right>}.  \]
 In particular, the first $n$ coefficients of $f_t$ is the first row of the $(1,n)$ principal block of $[f_a]\exp{(t-a)\left<zp\right>}$.
 \end{theorem}
 \begin{proof}  Since $f_t$ satisfies the Loewner partial differential equation almost everywhere, it also satisfies the formal Loewner partial differential equation almost everywhere.  By Proposition \ref{pr:formal_matrix_Loewner_equal} and Remark \ref{re:matrix_Loewner_ae} it follows that $[f_t]$ satisfies the matrix Loewner partial differential equation.  The solution with initial condition $[f_a]$ is clearly $[f_a]\exp{(t-a)\left<zp\right>}$.
 \end{proof}
 The same claim clearly holds for solutions to the Loewner ordinary differential equation.  For infinitesimal generators $p$ which are holomorphic in $\mathbb{D}$ and satisfy $\mbox{Re}(p)>0$, there is a solution $w_t$ to the Loewner ordinary differential equation which further satisfies $t \leq s \Rightarrow w_s(\mathbb{D}) \subset w_t(\mathbb{D})$.
 \begin{theorem}  \label{th:LoewnerODEsolved_by_exp}
  Let $h(z)=zp(z)$ be analytic on $\mathbb{D}$ and satisfy $\mbox{Re}(p) >0$.  Let $w_t$ be the solution to the Loewner ordinary differential equation
  \[  \frac{d}{dt} w_t(z)=-w_t(z) p(w_t(z))  \]
  on $[s,\infty)$ with initial condition $w_a(z)=w(z)$, where $w:\mathbb{D} \rightarrow \mathbb{D}$ is univalent and $w(0)=0$, $0<w'(0)<1$.  The power matrix of $w_t$ is given by
  \[  [w_t] = \left[\exp{-(t-a) \left<zp\right>}\right][w_a].  \]
  In particular, the first $n$ coefficients of $w_t$ is the first row of the exponential of the  $(1,n)$ principal block of $\left[\exp{-(t-a) \left<zp\right>}\right][w_a]$.
 \end{theorem}
 The proof is similar.
 \begin{remark}
  It is clear that Theorems \ref{th:LoewnerPDEsolved_by_exp} and \ref{eq:Loenwer_ODE_definition}
   will continue to hold whenever we have a solution to the Loewner equation, analytic near $0$, which is regular enough in $t$ that the power series of the solution can be differentiated term-by-term with respect to $t$.
  We will give a more precise statement of what it means to be a reasonable local solution to the Loewner partial differential equation in the next section.
 \end{remark}

 It is also evident that one can approximate a finite set of coefficients of a solution to the Loewner equation arbitrarily closely, with a finite matrix operation. Let $p$ be analytic on $\mathbb{D}$ and satisfy $\mbox{Re}(p)>0$.  Fix $n$.
 To keep notation simple, let $\left<zp\right>$ refer to the $(1,n)$-principal block of the infinitesimal generator matrix.  For any $n \times n$ matrix $A$ define the matrix norm
 \[  \| A \| = \sup_{1 \leq i,j \leq n} \left| A^i_j \right|.  \]
 In particular
 \begin{equation} \label{eq:zpnorm_by_coefficients}
  \| \left< zp \right> \| \leq \sup_{1 \leq l \leq n} n \left| \left< zp \right>^1_l \right|
 \end{equation}
 can be estimated in terms of the first $n$ coefficients of $p$.  By an elementary
 estimate
 \[  \|  \left(t\left<zp\right>\right)^m \| \leq t^m \|\left<zp\right>\|^m n^{m-1}  \]
 for all integers $m \geq 1$.
 Thus if we let
 \[  T_q = \sum_{m=0}^q \frac{t^m}{m!} \left<zp\right>^m  \]
 be the $q$th partial sum of the matrix exponential, we have the estimate
 \begin{equation} \label{eq:exponential_estimate}
  \| \exp{(t\left<zp\right>)} - T_q \| \leq \frac{1}{n} \sum_{m=q+1}^\infty \frac{t^mn^m}{m!} \left\| \left<zp\right>^m \right\| = \frac{1}{n} \left(
  \exp{tn \|\left<zp\right>\|} - \sum_{m=0}^q \frac{t^mn^m}{m!} \|  \left<zp \right>^m \| \right).
 \end{equation}
 We have thus just proven that
 \begin{theorem} \label{th:exp_polynomial_approximation}
  Let $f_t$ be a solution to the Loewner partial differential equation with constant infinitesimal generator $p$ and initial condition $f_a$
  as in Theorem \ref{th:LoewnerPDEsolved_by_exp}.  For any $T>a$ and $n \geq 0$,
  there is a $q$ so that for $t \in [s,T]$ the coefficients of the power series of $f_t$ are approximated by the first $n$ elements of the first row of
  \[  [f_0] \sum_{m=0}^q \frac{(t-a)^n}{m!} \left<zp \right>^m \]
  where $\left<zp\right>$ denotes the $(1,n)$th block of the power matrix of
  the infinitesimal generator.
  This approximation is uniform in the index of the coefficient and in $t$.

  Similarly, if $w_t$ is a solution to the Loewner ordinary differential equation with constant infinitesimal generator $p$ and initial condition $w_a$ as in Theorem \ref{th:LoewnerODEsolved_by_exp} then for any $T>a$ and $n>0$ there is a $q$ such that the coefficients of $w_t$ are approximated by
  \[  \left(\sum_{m=0}^q (-1)^m\frac{(t-a)^m}{m!} \left<zp \right>^m\right)[w_a]  \]
  uniformly on $[a,T]$ and in the index of the coefficient.
 \end{theorem}
 \begin{proof}
  We use the estimate
  \begin{align*}
    & \left\|[f_0] \left( \exp{(t-a)\left<zp\right>}  - \sum_{m=0}^q \frac{(t-a)^n}{m!} \left<zp \right>^m \right)
    \right\|\\
  & \ \ \ \ \ \ \ \ \ \ \ \ \leq   \|[f_0]\|  \left\| \left( \exp{(t-a)\left<zp\right>} - \sum_{m=0}^q \frac{(t-a)^n}{m!} \left<zp \right>^m  \right) \right\|
  \end{align*}
  and apply (\ref{eq:exponential_estimate}).  The proof of the ordinary differential equation case is similar.
 \end{proof}
 \begin{remark} It is clear that similar uniform approximations hold for coefficients of the positive and negative powers of the solution to the Loewner equation.
 \end{remark}
\end{subsection}
\begin{subsection}{The exponential map and holomorphic power series}
 We begin the section with what is now a simple observation.
 \begin{corollary}
  Let $f$ be holomorphic in a neighbourhood of $0$, and satisfy $f(0)=0$, $f'(0)\neq 0$.  There exists a formal power series $h(z)=zp(z)$ such that
  \[ [f]= \exp{\left<h\right>}.  \]
  In particular, there is a formal power series $f_t$ for all $t \in [0,1]$ such that the coefficients of $f_t$ are differentiable in $t$ and $f_t$ satisfies the
  formal Loewner partial differential equation
  \[  \dot{f}_t = zp(z)f_t'(z)  \]
  with $f_0(z)=z$ and $f_1(z)=f(z)$.
 \end{corollary}
 \begin{proof}
  This follows immediately from Propositions \ref{pr:formal_matrix_Loewner_equal} and \ref{pr:derivative_of_exp}.
 \end{proof}
 This leads to two natural questions.  First, if the generator is holomorphic, must its
 exponential be holomorphic?  Second, if a function is holomorphic, is it the exponential of a
 holomorphic power series?  These questions were posed by Huang \cite{Huang_private}, in the
 context of derivations and formal power series.   The main idea of this section is to phrase the
 question in terms of the power matrix and formal Loewner equations.  The connection to Loewner
 theory suggests some possible approaches to the question.

 In order to do this, we need a suitable notion of local solutions to the Loewner equations.
 As pointed out in Remark \ref{re:Loewner_natural}, the Loewner equations arise naturally even in the setting of formal power series, with any assumption of holomorphicity removed. Thus one is led to an
 intermediate version of the Loewner equation, in which the infinitesimal generator and solutions are holomorphic in some neighbourhood of the origin, but not necessarily on the entire unit disc?
 To this end we make the following definition.
 \begin{definition} \label{de:LoewnerPDE_solution_sense}
  Let $h_t(z)=zq_t(z)$ be analytic on a neighbourhood of $0$ for all $t \in (t_1,t_2)$.  We say that $f_t$ is a local solution of the Loewner PDE on $(t_1,t_2)$ if for all $t_0 \in (t_1,t_2)$, there is an interval $(a,b)$ containing $t_0$ and an open neighbourhood $U$ of $0$ such that $f_t$ and $q_t$ are complex analytic on $U$ for all $t \in (a,b)$, $f_t(z)$ is jointly continuous in $t$ and $z$ on $(a,b) \times U$,
  differentiable in $t$ for fixed $z$ and
  \[  \frac{d}{dt} f_t(z)=zq_t(z) f_t'(z). \]
  By a local solution $f_t$ of the Loewner partial differential equation on $[t_1,t_2)$ with initial condition $f$ we mean a solution of the form above, replacing with right continuity and right-hand derivatives in the appropriate places, with $f_{t_1}=f$.
 \end{definition}
 Similarly we define a notion of local solution to the ordinary differential equation.
 \begin{definition} \label{de:LoewnerODE_solution_sense}
  We say that $f_t$ is a local solution to the Loewner ordinary differential equation on $(t_1,t_2)$ if for all $t_0 \in (t_1,t_2)$, there is an interval $(a,b)$ containing $t_0$ and an open neighbourhood $U$ of $0$ such that $f_t$ and $q_t \circ f_t$ are complex analytic on $U$ for all $t \in (a,b)$, $f_t(z)$ is jointly continuous in $t$ and $z$ on $(a,b) \times U$, differentiable
  in $t$ for fixed $z$ and
  \[  \frac{d}{dt} f_t(z)= zq_t(z) f_t'(z). \]
 \end{definition}
 \begin{remark} \label{re:joint_continuity} It follows immediately from joint continuity that
 \[ \lim_{t \rightarrow s} f_t =f_s  \]
 uniformly on compact subsets of $U$ if $s \in (a,b)$ for $(a,b)$, $U$ as above.
 Thus the coefficients of the power series of $f_t$ are differentiable in $t$, and $f_t$ is a solution to the formal Loewner partial differential equation.
 \end{remark}
 \begin{remark}  If the infinitesimal generator $q_t$ is to vary, an appropriate definition of a local solution might involve weaker conditions on the $t$ dependence of $q$ and its solution (perhaps absolute continuity of the solution and measurability in $t$ of the generator, as is standard).  However, we are concerned with the case that $q$ constant in time; in this case, the solution must be jointly
 continuous, so our assumptions are not restrictive.
 \end{remark}

 It is now possible to pose the questions described above.
 Let $\mathcal{M}_A(0)$ denote the set of power matrices in $\mathcal{M}(0)$ arising from holomorphic power series.  Similarly let $\mathfrak{m}_A(0)$ denote the elements of $\mathfrak{m}(0)$ arising from holomorphic power series.
 \begin{question}  Is $\exp{\mathfrak{m}_A(0)} \subset \mathcal{M}_A(0)$?
 \end{question}
 We may also ask the corresponding question of local solutions to the Loewner partial differential equation.
 \begin{question}
  Given $h(z)=zp(z)$ for $p$ analytic in a neighbourhood of $0$, is there a local solution $f_t$ of the Loewner partial differential equation on $[0,\infty)$ with infinitesimal generator $p$ and initial condition $f_0(z)=z$?
 \end{question}

 Furthermore we have the two converse questions
 \begin{question} Is $\mathcal{M}_A(0) \subset \exp{\mathfrak{m}_A(0)}$?
 \end{question}
 and
 \begin{question} Let $f$ be holomorphic in a neighbourhood of $0$ and satisfy $f(0)=0$ and $f'(0)\neq 0$.  Is there an $h(z)=zp(z)$ which is analytic in a neighbourhood of $0$ and a $T$ such that the Loewner equation
 \[ \frac{d}{dt}f_t(z) = z p(z) f_t'(z)  \]
 has a local solution on $[0,T]$ with initial condition $f_0(z)=z$ such that $f_T=f$?
 \end{question}

 I conjecture that the answers to Questions 4.1 and 4.2 are ``no''.  However it is possible to demonstrate a partial result: the exponential of a holomorphic power series stays holomorphic on some finite interval in time.  This is a simple consequence of the Cauchy-Kowalevski theorem.
 \begin{theorem} \label{th:exp_analytic_smallt}  Let $h(z)=zp(z)$ and $g(z)$ be analytic in a neighbourhood $V$ of $0$.  There is an open set $U$ of $0$, an interval $(-T,T)$ and holomorphic functions $f_t$ for $t \in (-T,T)$ on $U$ such that
 \[  \frac{d}{dt} f_t(z) = zp(z) f_t'(z)  \]
 and $f_0(z)=g(z)$.
 \end{theorem}
 \begin{proof} It suffices to prove the theorem for $g(z)=z$, since if $f_t$ is a
 solution with initial condition $f_0(z)=z$, then $g \circ f_t$ is a solution
 with initial condition $g(z)$.    Let $z=x+iy$.
   Let $u(x,t)=f(x,t)-x$.  The restriction of
 the Loewner partial
 differential equation to the real line $y=0$ is equivalent to the following Cauchy
 problem for complex-valued $u$:
 \begin{equation} \label{eq:u_problem}
  \frac{\partial u}{\partial t}(x,t) = xp(x) \frac{\partial
  u}{\partial x}(x,t) + xp(x) \ \ \ \ u(x,t)=0.
 \end{equation}
 The function $xp(x)$ is real analytic on some interval
 containing $0$.
 By \cite[Theorem 1.41]{Folland} there exists a real analytic solution to this
 problem on a set $(-a,a) \times (-T,T)$ for some $a>0$.
 Since each term in equation (\ref{eq:u_problem}) is real analytic, we can
 substitute $z$ for $x$ to obtain a holomorphic function $u(z,t)$ satisfying
 the equation
 \begin{equation*}
  \frac{\partial u}{\partial t}(z,t) = zp(z) \frac{\partial
  u}{\partial z}(z,t) + zp(z) \ \ \ \ u(z,t)=0
 \end{equation*}
 on $\{ z\,:\,|z|<a \} \times (-T,T)$.  Thus the function
 $f(z,t)=u(z,t)+z$ is a solution to the Loewner partial
 differential equation on $\{ z\,:\,|z|<a \} \times (-T,T)$.
 \end{proof}

 Theorem \ref{th:exp_analytic_smallt} leads immediately to a partial answer to Question 1.
 \begin{corollary} For any $\left<h\right> \in \mathfrak{m}_A(0)$, there is a $T>0$ such that $\exp{\left(t \left<h\right>\right)} \subset \mathcal{M}_A(0)$ for all $t \in [0,T)$.
 \end{corollary}
 \begin{proof}
  By the previous theorem, there is a local solution $f_t$ to the Loewner partial differential equation on $[0,T)$ with initial condition $f_0(z)=z$ and infinitesimal generator $p$.  $f_t$ is thus also a solution to the formal Loewner partial differential equation.  The claim then follows from Propositions \ref{pr:formal_matrix_Loewner_equal} and \ref{pr:derivative_of_exp}.
 \end{proof}

 The Cauchy-Kowalevski theorem also shows that Questions 4.1 and 4.2 are equivalent.  Clearly a positive answer to Question 4.2 implies a positive answer to Question 4.1.  Conversely, assume that the answer to Question 4.1 is yes.  Fix an analytic $h(z)=zp(z)$ in a neighbourhood of $0$.  Fixing $t_0$, $\exp{(t_0 \left<zp\right>)}$ is the power matrix of an analytic function $f_{t_0}$.  Reasoning as in the proof of Theorem \ref{th:exp_analytic_smallt}, there is a local solution $f_t$ of the Loewner partial differential equation on $(t_0-a,t_0+a) \times U$ for some $a$ and initial condition $f_{t_0}$.
 Since this is true for any $t_0 \in [0,\infty)$ the answer to Question 4.2
 is also yes.
\begin{subsection}{The case that $p \in \mathcal{P}$}
 In this Section we make some observations regarding the case that
 $p \in \mathcal{P}$, the normalized analytic functions on
 the disc of positive real part (\ref{eq:Pdefinition}).  This case is contained in the
 standard theory of analytic semigroups \cite{Shoikhet}.  We make
 a few observations to place the above results on the power
 matrix in this context.

 A  partial answer to Question 4.1 is easily obtained from Loewner theory, if we make the assumption that the infinitesimal generator is an element of $\mathcal{P}$.
 Specifically, if $p \in \mathcal{P}$ then the matrix
exponential $\mbox{exp}\left<zp\right>$ is the power matrix of a
one-to-one map of the unit disc $\mathbb{D}$ into itself. Thus the
exponential can be forced to be holomorphic and univalent on the
disc by restricting the infinitesimal generator.
 \begin{theorem} \label{th:exp_P_bounded_univalent}
  If $p \in \mathcal{P}$ then $\exp{-t\left<zp\right>}$ is the power matrix of a bounded
  univalent map $f_t$ for all $t \in [0,\infty)$.  Furthermore, $f_t \circ f_s = f_{t+s}$
  and $f_s(\mathbb{D}) \subset f_t(\mathbb{D})$ whenever $s <t$.
 \end{theorem}
 \begin{proof} Assume that $p \in \mathcal{P}$.  Let $f_t$ be the solution of the Loewner equation
 \[  \frac{d}{dt} f_t(z)= -f_t(z) p \circ f_t(z).  \]
 Such a solution is guaranteed to exist for all $t$ by standard Loewner theory \cite{Pommerenke}.

 In particular, $f_t$ satisfies the formal Loewner ordinary differential equation so
 by Proposition \ref{pr:formal_matrix_Loewner_equal} $[f_t]$ satisfies the matrix Loewner partial differential equation.
 On the other hand by Proposition \ref{pr:derivative_of_exp}, the coefficients of  $\exp{-t \left<zp\right>}$ satisfy the same differential equation and initial conditions.  On each principal $(m,n)$ block, for $m\leq 1$ and $n \geq 1$, only the coefficients of $\left<zp\right>$, $[f_t]$ and $\exp{-\left<zp\right>}$ in that block enter the equation.  By the uniqueness of solutions to ordinary differential equations, the coefficients of $[f_t]$ and $\exp{-\left<zp\right>}$ are identical.
 \end{proof}
 \begin{remark}
 For $p_t \in \mathcal{P}$
 measurable in $t$, $t \in [0,\infty)$, it was shown by Friedland and Schiffer that
 the differential equation
 \begin{equation}  \label{eq:Friedland_Schiffer}
  \dot{f}_t = -zp_t(z) f_t'(z)
 \end{equation}
 with initial condition $f_0(z)=z$ has a
 solution on $\mathbb{D}$ almost everywhere in $t$.  (More
 precisely, they prove this if $p_t$ are extreme points of
 $\mathcal{P}$, but the proof goes through in general.  A full
 proof can be found in \cite{Olithesis}.  Friedland and Schiffer
 also consider more general initial conditions.)

 For $p_t$ independent of $t$, it can be shown that the solution of the
 Friedland-Schiffer equation must also satisfy the Loewner
 ordinary differential equation in a neighbourhood of $0$ with the
 same infinitesimal generator $p$.   This is a natural consequence of the semigroup formalism
 \cite{Shoikhet}.  The solutions must thus be the same on
 the entire disc.  In particular, the power matrices of the solutions $f_t$ to the
 Friedland-Schiffer equation
 (\ref{eq:Friedland_Schiffer}) with initial condition $f_0(z)=z$ must also satisfy
 $[f_t]=\exp{-t\left<zp\right>}$.

 On the level of power matrices, the fact that the solutions to the Loewner ODE and
 the Friedland-Schiffer equation starting at the
 identity must be the same is an
 obvious consequence of Proposition \ref{pr:derivative_of_exp}.
 \end{remark}

It is not true that every univalent holomorphic map from
$\mathbb{D}$ into $\mathbb{D}$ is the exponential of an element of
$\mathcal{P}$, as the following example shows.
\begin{example}
 Let $f:\mathbb{D} \rightarrow \mathbb{D}$ satisfy the
 normalizations $f(0)=0$, $f'(0)=e^{-T}=[f]^1_1$. By Theorem \ref{th:exp_invertible} there's an
 $h\in \coneounion$ such that $[f]=\exp{(-\left<h\right>)}$.
 Choose $p(z)=1+c_1 z + \cdots$ such that $h(z) = e^{-T}z p(z)$.
 Assume that $p \in \mathcal{P}$.  By Theorem
 \ref{th:exp_P_bounded_univalent},
 \[  [f_t] = \exp{-t\left<zp\right>}  \]
 is a solution to the Loewner ordinary differential equation with $f_t:\mathbb{D}
 \rightarrow \mathbb{D}$ for each $t$.  Thus $f_T=f$ is
 reachable in time $T$ by a solution to the Loewner equation with
 {\it constant} generator $p_t \equiv p$.

 By a result of Kufarev
 for every $C^1$ function $\kappa(t): [0,\infty)
 \rightarrow \mathbb{C}$ satisfying $|\kappa(t)|=1$, choosing
 \[  p_t(z) = \frac{1+\kappa(t)}{1-\kappa(t)} \]
 the solution $f_t$ to the ordinary Loewner equation with initial
 condition $f_0(z)=z$ is a univalent map from $\mathbb{D}$ onto
 $\mathbb{D}$, minus a single slit extending to the boundary of the
 disc.  On the other hand, given such a single slit map $f$ with derivative $f'(0)=e^{-T}$,
  by the Riemann mapping theorem
 there is a unique normalized one-parameter family $f_t$ such that
 $f_T=f$, $f_t(0)=0$, $f_t'(0)=e^{-t}$ and $f_t(\mathbb{D}) \subset
 f_s(\mathbb{D})$ whenever $t \geq s$.  Since the solution of
 the Loewner equation satisfies this property, there is thus a
 {\it unique} choice of $\kappa(t)$ (and hence $p_t(z)$) so that
 the solution of the Loewner equation has this terminal point at
 time $T$.  Of course $\kappa(t)$ need not be constant.
 We have thus exhibited a large class of examples of
 bounded univalent mappings which are {\it not} reachable by a
 Loewner chain with constant $p_t \equiv p$.  Therefore if $f$ is one of these mappings, $[f]\neq
 \exp{-T\left<zp \right>}$ for any $p \in \mathcal{P}$.
\end{example}
\end{subsection}

Finally we remark on an essential asymmetry between the ``outward''
and ``inward'' flows of univalent maps of the disc. The condition
$\mbox{Re}(p)>0$ specifies that the solutions of either the Loewner
ordinary differential equation
\[  \frac{d}{dt} f_t(z) = - f_t(z) p \circ f_t(z) \]
or the Friedland-Schiffer equation
\[  \frac{d}{dt} f_t(z) = - z p(z) f_t'(z)  \]
with initial condition $f_0(z)=z$ are {\it inward} flows, that is
$f_t(\mathbb{D}) \subset f_s(\mathbb{D})$ whenever $s \leq t$. With
the same restriction on $p$, the equation
\[  \frac{d}{dt} f_t(z) = f_t(z) p \circ f_t(z) \]
or the Loewner partial differential equation
\begin{equation} \label{eq:actual_Loewner_PDE}
 \frac{d}{dt} f_t(z) = z p(z) f_t'(z)
\end{equation}
with initial condition $f_0=f$ for some normalized univalent
function $f$ on $\mathbb{D}$, generate {\it outward} flows,
provided that solutions exist. There seems to be an
essential asymmetry between the outward and inward case
\cite{OliEricLoewnerHadamard}, which we will briefly describe
here.

If one requires that the solution to equation
(\ref{eq:actual_Loewner_PDE}) be univalent, it is very easy to
construct examples of initial functions $f_0$ and $p \in
\mathcal{P}$ such that the solution does not stay univalent on any
interval $[0,T)$ (see \cite[Example 1]{OliEricLoewnerHadamard}).
Furthermore, the solution might not even be holomorphic on
$\mathbb{D}$ even for a short time.  For example, if
\[  p(z)= z \, \frac{1 + z}{1-z}  \]
then the local solution $f_t(z)$ to the Loewner partial
differential equation with initial condition $f_0(z)=z$ is
$f_t(z)=k_0^{-1} \circ k_t(z)$ where
\[  k_t(z)=e^t \frac{z}{(1-z)^2}.  \]
The function $k_t$ maps $\mathbb{D}$ onto $\mathbb{C} \backslash
(-\infty,e^t/4]$.  Thus $f_t$ is not analytic on $\mathbb{D}$ for
any $t >0$.
\end{subsection}
\end{section}

\end{document}